\documentclass{tac}

\usepackage{xy}
\input diagxy

\usepackage[colorlinks=true]{hyperref}
\hypersetup{allcolors=[rgb]{0.1,0.1,0.4}}

\author{Rowan Poklewski-Koziell}

\address{Department of Mathematics and Applied Mathematics, University of Cape Town\\
 Rondebosch 7701
}

\title{A note on Frobenius-Eilenberg-Moore objects in dagger 2-categories}

\copyrightyear{2020}


\eaddress{PKLROW001@myuct.ac.za}

\newtheorem{theorem}{Theorem}

\newtheoremrm{rem}{Remark}

\mathrmdef{Hom}
\mathbfdef{Set}
\mathsfdef{DagCat}

\begin{document}

\maketitle
\begin{abstract}
We define Frobenius-Eilenberg-Moore objects for a dagger Frobenius monad in an arbitrary dagger 2-category, and extend to the dagger context a well-known universal property of the formal theory of monads. We show that the free completion of a 2-category under Eilenberg-Moore objects extends to the dagger context, provided one is willing to work with such dagger Frobenius monads whose endofunctor part suitably commutes with their unit. Finally, we define dagger lax functors and dagger lax-limits of such functors, and show that Frobenius-Eilenberg-Moore objects are examples of such limits.
\end{abstract}

% NOTE: it is good practice to \label all headings (and proclamations) immediately

\section{Preliminaries}\label{sec-Preliminaries}

A \emph{dagger category} $\mathbf{D}$ is a category equipped with a involutive functor $\dagger: \mathbf{D}^\text{op} \longrightarrow \mathbf{D}$ which is the identity on objects, called the \emph{dagger} of $\mathbf{D}$. A \emph{dagger functor} $F: \mathbf{D} \longrightarrow \mathbf{C}$ between dagger categories $\mathbf{D}$, $\mathbf{C}$ is a functor which commutes with the daggers on $\mathbf{D}$ and $\mathbf{C}$. A  $2$-category $\mathcal{D}$ is a \emph{dagger $2$-category} when each of the hom-categories $\mathcal{D}(A, B)$ are not only (small) categories, but dagger categories. More precisely, given vertically-composable $2$-cells $\alpha$ and $\beta$, and horizontally-composable $2$-cells $\sigma$ and $\theta$ in $\mathcal{D}$, the equalities
\begin{align*}
    (\alpha \cdot \beta)^\dagger = \beta^\dagger \cdot \alpha^\dagger \quad \quad \quad (\sigma * \theta)^\dagger = \sigma^\dagger * \theta^\dagger
\end{align*}
hold, where, here and elsewhere, $\cdot$ and $*$ denote the vertical and horizontal composition of $2$-cells, respectively, and where, as we shall do elsewhere, we have dropped all subscripts on daggers $\dagger$ to refer to particular hom-dagger-categories. The dagger $2$-category \DagCat of small dagger categories, dagger functors and natural transformations is a basic example. Given dagger $2$-categories $\mathcal{D}$, $\mathcal{C}$, a $2$-functor $F : \mathcal{D} \longrightarrow \mathcal{C}$ is a \emph{dagger $2$-functor} when for each pair of objects $D$, $D^\prime \in \mathcal{D}$, the functor
\begin{align*}
    F_{D, D^\prime} : \mathcal{D}(D, D^\prime) \longrightarrow \mathcal{C}(FD, FD^\prime) 
\end{align*}
is a dagger functor.

We shall say that a dagger $2$-category $\mathcal{D}$ is a \textit{full dagger sub-$2$-category} of $\mathcal{C}$ if there is a dagger $2$-functor $I: \mathcal{D} \longrightarrow \mathcal{C}$ such that for all objects $D$, $D^\prime$ of $\mathcal{D}$, the component dagger functor $I_{D, D^\prime} : \mathcal{D}(D, D^\prime) \longrightarrow \mathcal{C}(ID, ID^\prime)$ is an isomorphism of dagger categories. The weaker case of having $I_{D, D^\prime}$ only equivalences of categories which are \textit{unitarily essentially surjective} has no additional value in our work. The reader is encouraged to consult \cite[Chapter 3]{Kar19} for a more detailed account of such dagger equivalences.

If $(D, t)$ is a monad in a dagger $2$-category $\mathcal{D}$, it is obviously a comonad too. \cite{HK15, HK16} proposes that in a dagger $2$-category, the monads of interest are those that additionally satisfy the \emph{Frobenius law}.
\begin{definition}\cite{HK16}\label{daggerFrobMonad}
A monad $(D, t)$ (with multiplication $2$-cell $\mu: t^2 \longrightarrow t$ and unit $2$-cell $\eta: 1 \longrightarrow t$) in a dagger $2$-category $\mathcal{D}$ is a dagger Frobenius monad when the diagram
$$\bfig
    \node A(0,500)[t^2]
    \node B(750,500)[t^3]
    \node C(0,0)[t^3]
    \node D(750,0)[t^2]
    \arrow/->/[A`B;t\mu^\dagger]
    \arrow/->/[A`C;\mu^\dagger t]
    \arrow|r|/->/[B`D;\mu t]
    \arrow|b|/->/[C`D;t \mu]
\efig$$
commutes. Furthermore, $\mathsf{DFMnd}(\mathcal{D})$ is the dagger $2$-category in which:
\begin{itemize}
    \item $0$-cells are dagger Frobenius monads in $\mathcal{D}$;
    \item given $0$-cells $(A, s)$ and $(D, t)$, a $1$-cell $(f, \sigma): (A, s) \longrightarrow (D, t)$ consists of a $1$-cell $f: A \longrightarrow D$ and a $2$-cell $\sigma: tf \longrightarrow fs$ in $\mathcal{D}$, such that the diagrams:
    \begin{equation}\label{eq:morphOfMnd}
        \bfig
        \node TAL(0,650)[tfs]
        \node T2AR(800,650)[fss]
        \node P(-450,325)[ttf]
        \node T2AL(0,0)[tf]
        \node TAR(800,0)[fs]
        \arrow/->/[TAL`T2AR;\sigma s]
        \arrow/->/[P`TAL;t\sigma]
        \arrow|b|/->/[P`T2AL;\mu^tf]
        \arrow|r|/->/[T2AR`TAR;f\mu^s]
        \arrow|r|/->/[T2AL`TAR;\sigma]
        \node TAL2(2000,650)[fss]
        \node T2AR2(2800,650)[fs]
        \node P2(1550,325)[tfs]
        \node T2AL2(2000,0)[ttf]
        \node TAR2(2800,0)[tf]
        \arrow/->/[TAL2`T2AR2;f\mu^s]
        \arrow/->/[P2`TAL2;\sigma s]
        \arrow|b|/->/[P2`T2AL2;t\sigma^\dagger]
        \arrow|r|/->/[T2AR2`TAR2;\sigma^\dagger]
        \arrow|r|/->/[T2AL2`TAR2;\mu^tf]
        \efig
    \end{equation}
    \[\bfig
        \node A1IB1(675,-400)[tf]
        \node 1AI1B(1675,-400)[fs]
        \node AB(1175,-800)[f]
        \arrow/->/[A1IB1`1AI1B;\sigma]
        \arrow|l|/->/[AB`A1IB1;\eta^tf]
        \arrow|r|/->/[AB`1AI1B;f\eta^s]
        \efig\]
    commute, where $\mu^t: t^2 \longrightarrow t$ and $\mu^s: s^2 \longrightarrow s$ are the multiplications of $t$ and $s$, respectively, and $\eta^t: 1 \longrightarrow t$ and $\eta^s: 1 \longrightarrow s$ are the units of $t$ and $s$, respectively. Composition of $1$-cells is defined as $(g, \gamma)\cdot(f, \sigma) = (gf, g\sigma\cdot\gamma f)$;
    \item given $0$-cells $(A, s)$, $(D, t)$ and $1$-cells $(f, \sigma), (g, \gamma): (A, s) \longrightarrow (D, t)$ in $\mathsf{DFMnd}(\mathcal{D})$, a $2$-cell $\alpha: (f, \sigma) \longrightarrow (g, \gamma)$ in $\mathsf{DFMnd}(\mathcal{D})$ is a $2$-cell $\alpha : f \longrightarrow g$ in $\mathcal{D}$, such that the following diagrams
    $$\bfig
    \node A(0,500)[tf]
    \node B(750,500)[tg]
    \node C(0,0)[fs]
    \node D(750,0)[gs]
    \arrow/->/[A`B;t\alpha]
    \arrow/->/[A`C;\sigma]
    \arrow|r|/->/[B`D;\gamma]
    \arrow|b|/->/[C`D;\alpha s]
    \node A(1550,500)[tg]
    \node B(2300,500)[tf]
    \node C(1550,0)[gs]
    \node D(2300,0)[fs]
    \arrow/->/[A`B;t\alpha^\dagger]
    \arrow/->/[A`C;\gamma]
    \arrow|r|/->/[B`D;\sigma]
    \arrow|b|/->/[C`D;\alpha^\dagger s]
\efig$$
commute. Vertical and horizontal composition of $2$-cells is induced by the corresponding vertical and horizontal composition of $2$-cells in $\mathcal{D}$, as is the dagger on $2$-cells induced by the dagger on $2$-cells in $\mathcal{D}$.
\end{itemize}
There is an inclusion dagger $2$-functor $I: \mathcal{D} \longrightarrow \mathsf{DFMnd}(\mathcal{D})$, defined on $0$-cells by $I(D) = (D, 1)$, on $1$-cells by $I(f) = (f, 1)$, and on $2$-cells by $I(\alpha) = \alpha$.
\end{definition}

A dagger Frobenius monad in the dagger $2$-category \DagCat is of course simply a monad $(T, \mu, \eta)$ on a dagger category $\mathbf{D}$ whose endofunctor part $T$ is a dagger functor, and such that
\begin{align*}
    T(\mu_D) \mu^\dagger_{TD} = \mu_{TD} T(\mu^\dagger_D)
\end{align*}
for each $D$ in $\mathbf{D}$.

One may easily verify that any dagger Frobenius monad is a \emph{Frobenius monad} in the sense of \cite{Str04} -- however, neither that paper nor \cite{Lau} explore monads in the dagger context. In particular, algebras for these monads should satisfy an additional condition, so that they may behave quite differently from their non-dagger counterparts.

\begin{definition}
Let $T = (T, \mu, \eta)$ be a dagger Frobenius monad on a dagger category $\mathbf{D}$. A Frobenius-Eilenberg-Moore algebra (or FEM-algebra) for $T$ is an Eilenberg-Moore algebra $(D, \delta)$ for $T$, such that the diagram
    $$\bfig
        \node A(0,500)[T(D)]
        \node B(750,500)[T^2(D)]
        \node C(0,0)[T^2(D)]
        \node D(750,0)[T(D)]
        \arrow/->/[A`B;T(\delta^\dagger)]
        \arrow/->/[A`C;\mu^\dagger_D]
        \arrow|r|/->/[B`D;\mu_D]
        \arrow|b|/->/[C`D;T(\delta)]
    \efig$$
-- called the Frobenius law diagram for the algebra $(D, \delta)$ -- commutes. The class of all Frobenius-Eilenberg-Moore algebras and the class of all homomorphisms of Eilenberg-Moore algebras between FEM-algebras form a dagger category, which is denoted by $\mathsf{FEM}(\mathbf{D}, T)$.
\end{definition}

An \emph{adjunction} in a dagger $2$-category $\mathcal{D}$ is simply an adjunction in the underlying $2$-category.

For dagger $2$-categories $\mathcal{A}$, $\mathcal{D}$, there is $2$-category $[\mathcal{A}, \mathcal{D}]$, called the \emph{dagger 2-functor category}, consisting of dagger $2$-functors, $2$-natural transformations, and modifications. There is no need to specify ``dagger $2$-natural transformations'': given dagger $2$-functors $F, G: \mathcal{A} \longrightarrow \mathcal{D}$ in $[\mathcal{A}, \mathcal{D}]$, a $2$-natural transformation is a family $\phi = \big(\phi_A: FA \longrightarrow GA\big)_{A \in \mathcal{A}}$ of $1$-cells in $\mathcal{D}$, such that the diagram
    $$\bfig
        \node A(0,500)[\mathcal{A}(A, B)]
        \node B(1500,500)[\mathcal{D}(FA, FB)]
        \node C(0,0)[\mathcal{D}(GA, GB)]
        \node D(1500,0)[\mathcal{D}(FA, GB)]
        \arrow/->/[A`B;F_{A, B}]
        \arrow/->/[A`C;G_{A, B}]
        \arrow|r|/->/[B`D;\mathcal{D}(FA, \phi_{B})]
        \arrow|b|/->/[C`D;\mathcal{D}(\phi_{A}, GB)]
    \efig$$
commutes for all objects $A$, $B$ in $\mathcal{D}$, and clearly the representable functors $\mathcal{D}(FA, \phi_{B})$ and $\mathcal{D}(\phi_{A}, GB)$ of this diagram are of course dagger functors themselves. The dagger structure on $\mathcal{D}$ then naturally induces a dagger structure on $[\mathcal{A}, \mathcal{D}]$.

A dagger $2$-functor $F: \mathcal{D} \longrightarrow \mathsf{DagCat}$ is \emph{representable}, when there is some $D$ in $\mathcal{D}$ and an isomorphism $\phi: \mathcal{D}(D, -) \longrightarrow F$ in $[\mathcal{D}, \mathsf{DagCat}]$. The pair $(D, \phi)$ is called a \emph{representation} of $F$. What is worth remarking is that, for a dagger $2$-category $\mathcal{C}$ and a dagger $2$-functor $R: \mathcal{D} \longrightarrow \mathcal{C}$, when, for each object $C$ of $\mathcal{C}$, the dagger $2$-functor $\mathcal{C}(C, R-): \mathcal{D} \longrightarrow \mathsf{DagCat}$ is representable -- with representation $(LC, \phi_C)$ -- one has that the unique (up to $2$-natural isomorphism) $2$-functor $L : \mathcal{C} \longrightarrow \mathcal{D}$ such that 
    $$\bfig
        \node A(0,0)[\mathcal{D}(LC, D)]
        \node B(750,0)[\mathcal{C}(C, RD)]
        \arrow/->/[A`B;\phi_{C, D}]
    \efig$$
is $2$-natural in both $C$ and $D$, is also a dagger $2$-functor. This is easily seen from the standard construction of $L$, as displayed in, say, \cite[Section 1.10]{Kel}. Furthermore, $L$ is of course the left $2$-adjoint of $R$ and such $2$-adjunctions correspond bijectively to $2$-natural isomorphisms $\phi$ in the above display.

Finally, one also has a Yoneda Lemma for dagger $2$-categories: there are dagger $2$-functors $E$, $N :[\mathcal{D}^\text{op}, \mathsf{DagCat}] \times \mathcal{D} \longrightarrow \mathsf{DagCat}$, given, respectively, on $0$-cells by $E(F, D) = F(D)$ and $N(F, D) = [\mathcal{D}^\text{op}, \mathsf{DagCat}](\mathcal{D}(-, D), F)$ and, furthermore, an isomorphism $y: N \longrightarrow E$.

Dagger Frobenius monads and categories of Frobenius-Eilenberg-Moore algebras for such monads were first considered in \cite{HK15} and \cite{HK16}, in which they are shown to include the important example of quantum measurements. In this paper, we continue work initiated in those papers in pursuit of a \emph{formal theory of dagger Frobenius monads} in the spirit of \cite{Str72} and \cite{LS02}.

\section{Frobenius-Eilenberg-Moore objects}\label{sec-FEM-results}

Let $(D, t)$ (with multiplication and unit given, respectively, by $\mu$ and $\eta$) be a dagger Frobenius monad in a dagger $2$-category $\mathcal{D}$. Then $\big(\mathcal{D}(A, D), \mathcal{D}(A, t)\big)$ is a dagger Frobenius monad (with multiplication and unit given, respectively, by $\mathcal{D}(A, \mu)$ and $\mathcal{D}(A, \eta)$) in $\mathsf{DagCat}$, for every object $A$ of $\mathcal{D}$. We may now construct the dagger category of Frobenius-Eilenberg-Moore algebras $\mathsf{FEM}\big(\mathcal{D}(A, D), \mathcal{D}(A, t)\big)$ for the dagger Frobenius monad $\mathcal{D}(A, t)$ on the dagger category $\mathcal{D}(A, D)$. Applying these observations to the case $\mathcal{D} = \mathsf{DagCat}$, we arrive at the following result for a dagger category $\mathbf{D}$ and a dagger Frobenius monad $(T, \mu, \eta)$ on $\mathbf{D}$.

\begin{proposition}\label{prop1}
Suppose $F: \mathbf{A} \longrightarrow \mathbf{D}$ is a dagger functor, $(T, \mu, \eta)$ is a dagger Frobenius monad on the dagger category $\mathbf{D}$, and $\sigma: TF \longrightarrow F$ is a natural transformation. $(FA, \sigma_A)_{A \in \mathbf{A}}$ is a family of Frobenius-Eilenberg-Moore algebras for $T$ if and only if $(F, \sigma)$ is a Frobenius-Eilenberg-Moore algebra for the dagger Frobenius monad $\mathsf{DagCat}(\mathbf{A}, T)$ on the dagger category $\mathsf{DagCat}(\mathbf{A}, \mathbf{D})$. Furthermore, given another such Frobenius-Eilenberg-Moore algebra $(G, \gamma)$ for $\mathsf{DagCat}(\mathbf{A}, T)$, and a natural transformation $\alpha: F \longrightarrow G$, the family $\big(\alpha_A: (FA, \sigma_A) \longrightarrow (GA, \gamma_A)\big)_{A \in \mathbf{A}}$ is a family of Eilenberg-Moore algebra homomorphisms if and only if $\alpha: (F, \sigma) \longrightarrow (G, \gamma)$ is a homomorphism of Eilenberg-Moore algebras for the monad $\mathsf{DagCat}(\mathbf{A}, T)$.
\end{proposition}

\begin{proof}
A routine calculation shows that, for every object $A$ in $\mathbf{A}$, the diagram
    $$\bfig
        \node A(0,500)[T^2(FA)]
        \node B(1000,500)[T(FA)]
        \node C(0,0)[T(FA)]
        \node D(1000,0)[FA]
        \node E(1733,500)[FA]
        \arrow/->/[A`B;\mu_{FA}]
        \arrow/->/[A`C;T(\sigma_A)]
        \arrow|r|/->/[B`D;\sigma_A]
        \arrow|b|/->/[C`D;\sigma_A]
        \arrow/->/[E`B;\eta_{FA}]
        \arrow|b|/=/[E`D; ]
    \efig$$
commutes if and only if the diagram
    $$\bfig
        \node A(0,500)[\mathsf{DagCat}(\mathbf{A}, T)^2(F)]
        \node B(1500,500)[\mathsf{DagCat}(\mathbf{A}, T)(F)]
        \node C(0,0)[\mathsf{DagCat}(\mathbf{A}, T)(F)]
        \node D(1500,0)[F]
        \node E(2600,500)[F]
        \arrow/->/[A`B;\mathsf{DagCat}(\mathbf{A}, \mu)(F)]
        \arrow/->/[A`C;\mathsf{DagCat}(\mathbf{A}, T)(\sigma)]
        \arrow|r|/->/[B`D;\sigma]
        \arrow|b|/->/[C`D;\sigma]
        \arrow/->/[E`B;\mathsf{DagCat}(\mathbf{A}, \eta)(F)]
        \arrow|b|/=/[E`D; ]
    \efig$$
commutes. That is, the family $(FA, \sigma_A)_{A \in \mathbf{A}}$ is a family of Eilenberg-Moore algebras for $T$ if and only if $(F, \sigma)$ is an Eilenberg-Moore algebra for the monad $\mathsf{DagCat}(\mathbf{A}, T)$ on the (dagger) category $\mathsf{DagCat}(\mathbf{A}, \mathbf{D})$. Likewise, for every object $A$ in $\mathbf{A}$, the diagram
    $$\bfig
        \node A(0,500)[T(FA)]
        \node B(1200,500)[T^2(FA)]
        \node C(0,0)[T^2(FA)]
        \node D(1200,0)[T(FA)]
        \arrow/->/[A`B;T(\sigma^\dagger_{A})]
        \arrow/->/[A`C;\mu^\dagger_{FA}]
        \arrow|r|/->/[B`D;\mu_{FA}]
        \arrow|b|/->/[C`D;T(\sigma_{A})]
    \efig$$
commutes if and only if the diagram
    $$\bfig
        \node A(0,500)[\mathsf{DagCat}(\mathbf{A}, T)(F)]
        \node B(1500,500)[\mathsf{DagCat}(\mathbf{A}, T)^2(F)]
        \node C(0,0)[\mathsf{DagCat}(\mathbf{A}, T)^2(F)]
        \node D(1500,0)[\mathsf{DagCat}(\mathbf{A}, T)(F)]
        \arrow/->/[A`B;\mathsf{DagCat}(\mathbf{A}, T)(\sigma^\dagger)]
        \arrow/->/[A`C;\mathsf{DagCat}(\mathbf{A}, \mu)^\dagger(F)]
        \arrow|r|/->/[B`D;\mathsf{DagCat}(\mathbf{A}, \mu)(F)]
        \arrow|b|/->/[C`D;\mathsf{DagCat}(\mathbf{A}, T)(\sigma)]
    \efig$$
commutes. The second part of the proposition is similarly proved.
\end{proof}

\begin{theorem}\label{th1}
Suppose $(T, \mu, \eta)$ is a dagger Frobenius monad on the dagger category $\mathbf{D}$. For every dagger category $\mathbf{A}$, there is an isomorphism of dagger categories
\begin{align*}
        \mathsf{DagCat}(\mathbf{A}, \mathsf{FEM}(\mathbf{D}, T)) \cong \mathsf{FEM}(\mathsf{DagCat}(\mathbf{A}, \mathbf{D}), \mathsf{DagCat}(\mathbf{A}, T))
\end{align*}
which is $2$-natural in each of the arguments.
\end{theorem}

\begin{proof}
Each dagger functor $\overline{F}: \mathbf{A} \longrightarrow \mathsf{FEM}(\mathbf{D}, T)$ determines a dagger functor $F = U^T\overline{F} : \mathbf{A} \longrightarrow \mathbf{D}$ and a family $\big(FA, \sigma_A\big)_{A \in \mathbf{A}}$ of FEM-algebras, where $U^T: \mathsf{FEM}(\mathbf{D}, T) \longrightarrow \mathbf{D}$ is the forgetful (dagger) functor. Since $\overline{F}$ is a functor, the family $\sigma = \big(\sigma_A : TFA \longrightarrow FA\big)$ is a natural transformation $\sigma: TF \longrightarrow F$. Therefore, by Proposition \ref{prop1}, $(F, \sigma)$ is a FEM-algebra for the dagger Frobenius monad $\mathsf{DagCat}(\mathbf{A}, T)$ on the dagger category $\mathsf{DagCat}(\mathbf{A}, \mathbf{D})$.

Conversely, given a dagger functor $F: \mathbf{A} \longrightarrow \mathbf{D}$ and a natural transformation $\sigma: TF \longrightarrow F$ such that $(F, \sigma)$ is a FEM-algebra for the dagger Frobenius monad $\mathsf{DagCat}(\mathbf{A}, T)$, for each object $A$ of $\mathbf{A}$, $\big(FA, \sigma_A\big)$ is a FEM-algebra for $T$, again by Proposition \ref{prop1}. Since $\sigma: TF \longrightarrow F$ is a natural transformation, for each morphism $f: A \longrightarrow B$ of $\mathbf{A}$, $Ff: FA \longrightarrow F(B)$ is a morphism $\big(FA, \sigma_A\big) \longrightarrow \big(FB, \sigma_{B}\big)$ of Eilenberg-Moore algebras. This now defines a functor $\overline{F}: \mathbf{A} \longrightarrow \mathsf{FEM}(\mathbf{D}, T)$. 

Next, the second part of Proposition \ref{prop1} similarly establishes correspondences between natural transformations $\overline{F} \longrightarrow \overline{G}$ and homomorphisms $(F, \sigma) \longrightarrow (G, \gamma)$ of Eilenberg-Moore algebras for the monad $\mathsf{DagCat}(\mathbf{A}, T)$, which preserve daggers.

Clearly, these correspondences are inverses of each other. It is routine to show that each is $2$-natural in each of the arguments.
\end{proof}

The previous theorem suggests our main definition.

\begin{definition}\label{FEMobj}
For a dagger $2$-category $\mathcal{D}$, a dagger Frobenius monad $(D, t)$ in $\mathcal{D}$ is said to have a Frobenius-Eilenberg-Moore object (or FEM-object) if the dagger $2$-functor
\begin{align*}
    \mathsf{FEM}\big(\mathcal{D}(-, D), \mathcal{D}(-, t)\big) : \mathcal{D}^{\text{op}} \longrightarrow \mathsf{DagCat}
\end{align*}
whose object-part is defined by $A \longmapsto \mathsf{FEM}\big(\mathcal{D}(A, D), \mathcal{D}(A, t)\big)$, is representable. A choice of a representing object in $\mathcal{D}$, denoted $\mathsf{FEM}(D, t)$, is called the Frobenius-Eilenberg-Moore object for $(D, t)$. $\mathcal{D}$ is further said to have Frobenius-Eilenberg-Moore objects if every dagger Frobenius monad $(D, t)$ in $\mathcal{D}$ has a Frobenius-Eilenberg-Moore object.
\end{definition}

\begin{proposition}\label{prop2}
Suppose $(D, t)$ is a dagger Frobenius monad in the dagger $2$-category $\mathcal{D}$. For every object $A$ of $\mathcal{D}$, there is an isomorphism of dagger categories
\begin{align}\label{monadsToRep}
    \mathsf{DFMnd}(\mathcal{D})((A, 1), (D, t)) \cong \mathsf{FEM}(\mathcal{D}(A, D), \mathcal{D}(A, t))
\end{align}
$2$-natural in each of the arguments.
\end{proposition}

\begin{proof}
One easily shows that, to give a pair $(f, \sigma)$ in which $f: A \longrightarrow D$ is a $1$-cell and $\sigma: tf \longrightarrow f$ a $2$-cell in $\mathcal{D}$ satisfying the top-left and bottom diagrams (\ref{eq:morphOfMnd}) for the monads $(A, 1)$ and $(D, t)$ is exactly to give an Eilenberg-Moore algebra for the monad $\mathcal{D}(A, t)$ on $\mathcal{D}(A, D)$. $(F, \sigma)$ is, moreover, a morphism of dagger Frobenius monads $(A, 1) \longrightarrow (D, t)$, exactly when, by the top-right diagram (\ref{eq:morphOfMnd}), $\sigma^\dagger \cdot \sigma = \mu f \cdot t \sigma^\dagger$, which is the statement that $\sigma^\dagger: (f, \sigma) \longrightarrow (tf, \mu f) = (\mathcal{D}(A, t)(f), \mathcal{D}(A, \mu)(f))$ is a homomorphism of Eilenberg-Moore algebras for the monad $\mathcal{D}(A, t)$. By \cite[Lemma 6.8]{HK16}, this is exactly to say that $(f, \sigma)$ is a FEM-algebra for the dagger Frobenius monad $\mathcal{D}(A, t)$.

Finally, for a second morphism $(g, \gamma) : (A, 1) \longrightarrow (D, t)$ of dagger Frobenius monads, to give a $2$-cell $\alpha: (f, \sigma) \longrightarrow (g, \gamma)$ in $\mathsf{DFMnd}(\mathcal{D})$ is exactly to give a homomorphism $(f, \sigma) \longrightarrow (g, \gamma)$ of Eilenberg-Moore algebras for the monad $\mathcal{D}(A, t)$, by \cite[Lemma 6.7]{HK16}.
\end{proof}

\begin{definition}\cite{HK16}
A dagger $2$-category $\mathcal{D}$ admits the construction of Frobenius-Eilenberg-Moore algebras when the inclusion dagger $2$-functor $I: \mathcal{D} \longrightarrow \mathsf{DFMnd}(\mathcal{D})$ has a right $2$-adjoint, which is denoted $\mathsf{FEM}: \mathsf{DFMnd}(\mathcal{D}) \longrightarrow \mathcal{D}$.
\end{definition}

From Proposition \ref{prop2}, the following result is immediate.

\begin{theorem}\label{thm2}
A dagger $2$-category $\mathcal{D}$ admits the construction of Frobenius-Eilenberg-Moore algebras if and only if $\mathcal{D}$ has Frobenius-Eilenberg-Moore objects. In particular, to give a right adjoint to $I: \mathcal{D} \longrightarrow \mathsf{DFMnd}(\mathcal{D})$ is precisely to give a choice, for each dagger Frobenius monad in $\mathcal{D}$ of a Frobenius-Eilenberg-Moore-object.
\end{theorem}

\noindent Theorems \ref{th1} and \ref{thm2} now give the following known result.

\begin{corollary}\cite[Theorem 7.5]{HK16}\label{hk16}
$\mathsf{DagCat}$ admits the construction of Frobenius-Eilenberg-Moore algebras.
\end{corollary}

When a dagger Frobenius monad $(D, t)$ in $\mathcal{D}$ has a FEM-object, the dagger isomorphism (\ref{monadsToRep}) uniquely determines a morphism of dagger Frobenius monads $(u^t, \xi): \big(\mathsf{FEM}(D, t), 1\big)$ $\longrightarrow \big(D, t\big)$, in which we think of the $1$-cell $u^t$ as the ``forgetful'' $1$-cell. Moreover, if $\mathcal{D}$ further admits the construction of Frobenius-Eilenberg-Moore algebras, then the component of the counit of the $2$-adjunction evaluated at the dagger Frobenius monad $(D, t)$ is $(u^t, \xi)$. In particular, in the case that $\mathcal{D} = \mathsf{DagCat}$, the forgetful $1$-cell $U^T$ is of course the usual forgetful dagger functor $\mathsf{FEM}(\mathbf{D}, T) \longrightarrow \mathbf{D}$.

\cite{Str72} shows that much of the $1$-dimensional theory of monads can be described by several important universal properties in a $2$-dimensional context. We next show that in passing to the dagger context, there are corresponding universal properties.

\begin{lemma}\label{lem1}
For an adjunction $f \dashv u$ in a dagger $2$-category $\mathcal{D}$, the monad generated by the adjunction $f \dashv u$ is a dagger Frobenius monad.
\end{lemma}

\begin{proof}
If $f \dashv u$ is an adjunction in a dagger $2$-category $\mathcal{D}$, with counit $\epsilon : fu \longrightarrow 1$ and unit $\eta: 1 \longrightarrow uf$, then we also have $u \dashv f$, with counit $\eta^\dagger: uf \longrightarrow 1$ and unit $\epsilon^\dagger: 1 \longrightarrow fu$. \cite[Corollary 2.22]{Lau} now says that the monad $(D, uf)$ generated by the adjunction $f \dashv u$ is a dagger Frobenius monad.
\end{proof}

Following this proposition we call $(D, uf)$ the \textit{dagger Frobenius monad generated by the adjunction $f \dashv u$}.

\begin{theorem}\cite[Theorem 7.4]{HK16}\label{universal1}
Every dagger Frobenius monad in a dagger $2$-category $\mathcal{D}$ having a Frobenius-Eilenberg-Moore object is generated by an adjunction.
\end{theorem}

When a dagger Frobenius monad $(D, t)$ in a dagger $2$-category $\mathcal{D}$ has a FEM-object, the isomorphism of dagger categories
\begin{align}\label{isoOfDag}
    \mathcal{D}\big(A, \mathsf{FEM}(D, t)\big) \longrightarrow \mathsf{DFMnd}(\mathcal{D})\big((A, 1), (D, t)\big)
\end{align}
is defined by $f \longmapsto (u^tf, \xi f)$ on $1$-cells and $\sigma \longmapsto u^t\sigma$ on $2$-cells, for the unique morphism $(u^t, \xi): \big(\mathsf{FEM}(D, t), 1\big) \longrightarrow (D, t)$ of dagger Frobenius monads. The proof of Theorem \ref{universal1} shows that, for a dagger Frobenius monad $(D, t)$ in a dagger $2$-category $\mathcal{D}$, if $(D, t)$ has a FEM-object, there exists a unique $1$-cell $f^t : D \longrightarrow \mathsf{FEM}(D, t)$ such that $t = u^t f^t$ and $\mu = \xi f^t$, and a unique $2$-cell $\epsilon^t : f^t u^t \longrightarrow 1$ such that $u^t \epsilon^t = \xi$. Furthermore, $f^t$ is a left adjoint of $u^t$ and generates the dagger Frobenius monad $(D, t)$.

\begin{theorem}\label{universal2}
In the notation above, suppose the dagger Frobenius monad $(D, t)$ generated by the adjunction $f \dashv u$ has a Frobenius-Eilenberg-Moore object. Then, there exists a unique $1$-cell $n: A \longrightarrow \mathsf{FEM}(D, t)$ such that $u^t n = u$ and $u \epsilon = \xi n$, where $\epsilon$ is the counit of the adjunction $f \dashv u$. Moreover, this $n$ satisfies $nf = f^t$ and $n \epsilon = \epsilon^t n$.
\end{theorem}

\begin{proof}
One easily verifies that $(u, u \epsilon) : (A, 1) \longrightarrow (D, t)$ is a morphism of monads. It remains to verify that it is a morphism of dagger Frobenius monads. From the top-right diagram of (\ref{eq:morphOfMnd}), $(u, u\epsilon)$ is a morphism of dagger Frobenius monads if and only if 
\begin{align*}
u(\epsilon^\dagger \cdot \epsilon) = u\epsilon^\dagger \cdot u \epsilon = u\epsilon fu \cdot ufu\epsilon^\dagger = u(\epsilon fu \cdot fu\epsilon^\dagger)
\end{align*}
But, a straightforward application of the interchange law gives the equalities
\begin{align*}
     fu\epsilon \cdot \epsilon^\dagger fu = \epsilon^\dagger \cdot \epsilon = \epsilon fu \cdot fu \epsilon^\dagger
\end{align*}
And so, $(u, u\epsilon)$ is indeed a morphism of dagger Frobenius monads. The rest of the proof proceeds identically to the similar proof in \cite{Str72}. Since $(u, u\epsilon) : (A, 1) \longrightarrow (D, t)$ is a morphism of dagger Frobenius monads, there exists a unique $1$-cell $n: A \longrightarrow \mathsf{FEM}(D, t)$ such that the diagram
\[\bfig
    \node A(0,0)[(D, t)]
    \node B(-500,500)[(A, 1)]
    \node C(500,500)[\big(\mathsf{FEM}(D, t), 1 \big)]
    \arrow/->/[B`C;(n, 1)]
    \arrow|b|/->/[B`A;(u, u\epsilon)]
    \arrow|r|/->/[C`A;(u^t, \xi)]
\efig\]
commutes. Therefore, $u^t n = u$ and $u \epsilon = \xi n$, so that $u^t(n \epsilon) = u \epsilon = \xi n = u^t(\epsilon^t n)$. Therefore, by the dagger isomorphism (\ref{isoOfDag}), we have $n \epsilon = \epsilon^t n$. Finally,
\begin{align*}
    u^t(nf) = uf = t \\
    \xi(nf) = u \epsilon f = \mu
\end{align*}
By the property which uniquely determines $f^t$, we have $f^t = nf$.
\end{proof}

Since $\mathsf{DagCat}$ admits the construction of Frobenius-Eilenberg-Moore algebras, the following result is immediate.

\begin{corollary}\cite[Theorem 6.9]{HK16}
Suppose $F$ and $U$ are dagger adjoints between dagger categories $\mathbf{A}$ and $\mathbf{D}$, with $T$ the dagger Frobenius monad generated by $F \dashv U$. Then, there exists a unique dagger functor $N: \mathbf{A} \longrightarrow \mathsf{FEM}\big(\mathbf{D}, T\big)$ such that $U^T N = U$ and $NF = F^T$.
\end{corollary}

\begin{definition}
The unique $1$-cell $n : A \longrightarrow \mathsf{FEM}(D, t)$ of Theorem \ref{universal2} is called the right comparison $1$-cell of the adjunction $f \dashv u$. If this $1$-cell is a dagger equivalence (that is, there is a $1$-cell $m : \mathsf{FEM}(D, t) \longrightarrow A$, and $2$-cell unitaries $nm \cong 1$ and $1 \cong mn$), then the adjunction $f \dashv u$ is said to be monadic.
\end{definition}

Note that $2$-functors between $2$-categories send adjunctions to adjunctions. The formulation of Frobenius-Eilenberg-Moore objects as representing objects for a representable dagger $2$-functor in the previous section now gives our final important result, whose proof is identical to that of \cite[Corollary 8.1]{Str72}.

\begin{corollary}
Suppose the dagger Frobenius monad generated by an adjunction $f \dashv u$ in a dagger $2$-category $\mathcal{D}$ has a Frobenius-Eilenberg-Moore object. The adjunction $f \dashv u$ is monadic if and only if, for each object $X$ of $\mathcal{D}$, the adjunction $\mathcal{D}(X, f) \dashv \mathcal{D}(X, u)$ in $\mathsf{DagCat}$ is monadic.
\end{corollary}

\section{Free completions under FEM-objects}

Dually, we define \textit{Frobenius-Kleisli objects} for dagger Frobenius monads in a dagger $2$-category.

\begin{definition}
A Frobenius-Kleisli object for a dagger Frobenius monad $(D, t)$ in a dagger $2$-category $\mathcal{D}$ is a Frobenius-Eilenberg-Moore object for $(D, t)$ considered as a dagger Frobenius monad in $\mathcal{D}^{\text{op}}$. A Frobenius-Kleisli object for $(D, t)$, when it exists, is denoted by $\mathsf{FK}(D, t)$, and in particular satisfies, for each object $X$ in $\mathcal{D}$, the following isomorphism of dagger categories
\begin{align*}
    \mathcal{D}\big(\mathsf{FK}(D, t), X\big) \cong \mathsf{FEM}\big(\mathcal{D}(D, X), \mathcal{D}(t, X)\big)
\end{align*}
$2$-natural in each of the arguments. $\mathcal{D}$ is said to have Frobenius-Kleisli objects if every dagger Frobenius monad in $\mathcal{D}$ has a Frobenius-Kleisli object.
\end{definition}

From \cite[Lemma 6.1]{HK16} we know that the Kleisli category $\mathbf{D}_{T}$ for a dagger Frobenius monad $(T, \mu, \eta)$ on a dagger category $\mathbf{D}$ carries a canonical dagger structure, given by
\begin{align}{\label{daggerOnFK}}
    \big(f: C \longrightarrow TD\big) &\longmapsto \big(T(f^\dagger)\mu^\dagger_D \eta_D : D \longrightarrow TC\big)
\end{align}
which commutes with the canonical dagger functors $\mathbf{D} \longrightarrow \mathbf{D}_{T}$ and $\mathbf{D}_{T} \longrightarrow \mathbf{D}$.
In fact, this makes $\mathbf{D}_{T}$ a Frobenius-Kleisli object for $(\mathbf{D}, T)$.

\begin{theorem}\label{DagCatHasFK}
Each dagger Frobenius monad $T = (T, \mu, \eta)$ on a dagger category $\mathbf{D}$ has a Frobenius-Kleisli object, which is the Kleisli category $\mathbf{D}_{T}$ of the monad $T$.
\end{theorem}

\begin{proof}
Let $F^{T}: \mathbf{D} \longrightarrow \mathsf{FEM}(\mathbf{D}, T)$ and $F_{T} : \mathbf{D} \longrightarrow \mathbf{D}_{T}$ denote the canonical free (dagger) functors. For a dagger category $\mathbf{X}$ and a dagger functor $S^\prime : F^T(\mathbf{D}) \longrightarrow \mathbf{X}$, the pair $(S^\prime F^T, S^\prime \mu)$ is a Frobenius-Eilenberg-Moore algebra for the dagger Frobenius monad $\mathsf{DagCat}(T, \mathbf{X})$ on the dagger category $\mathsf{DagCat}(\mathbf{D}, \mathbf{X})$. Indeed, since for each object $D$ in $\mathbf{D}$, $F^T(D) = \big(T(D), \mu_D\big)$ is an Eilenberg-Moore algebra for the monad $T$, $(S^\prime F^T, S^\prime \mu)$ is surely an Eilenberg-Moore algebra for $\mathsf{DagCat}(T, \mathbf{X})$. Furthermore, since $T$ is a dagger Frobenius monad, $(S^\prime F^T, S^\prime \mu)$ is additionally a Frobenius-Eilenberg-Moore algebra. Sending a natural transformation $\alpha: S^\prime \longrightarrow S^{\prime\prime} : F^T(\mathbf{D}) \longrightarrow \mathbf{X}$ to the homomorphism $\alpha F^T : S^{\prime}F^T \longrightarrow S^{\prime\prime}F^T$ of Eilenberg-Moore algebras then determines a dagger functor
\begin{align}\label{dagger-iso-left}
    \mathsf{DagCat}\big(F^T(\mathbf{D}), \mathbf{X}\big) \longrightarrow \mathsf{FEM}\big(\mathsf{DagCat}(\mathbf{D}, \mathbf{X}), \mathsf{DagCat}(T, \mathbf{X})\big)
\end{align}
On the other hand, if $(S, \phi)$ is a Frobenius-Eilenberg-Moore algebra for the dagger Frobenius monad $\mathsf{DagCat}(T, \mathbf{X})$, the mappings
\begin{align*}
    D &\longmapsto SD,\\
    \big(f: C \longrightarrow TD\big) &\longmapsto \big(\phi_D Sf : SC \longrightarrow SD\big)
\end{align*}
yield a dagger functor $\overline{S}: \mathbf{D}_{T} \longrightarrow \mathbf{X}$. For, given morphisms $g: B \longrightarrow TC$ and $f: C \longrightarrow TD$ in $\mathbf{D}$, the composite morphism $f\cdot g$ in $\mathbf{D}_{T}$ is given by the morphism $\mu_D T(f) g: B \longrightarrow TD$ in $\mathbf{D}$, and so
\begin{align*}
    \overline{S}(f \cdot g) & = \phi_D S(\mu_D) ST(f) S(g) \\
    & = \phi_D \phi_{T(C)} ST(f) S(g) \\
    & = \phi_D S(f)\phi_C S(g) = \overline{S}(f) \overline{S} (g)
\end{align*}
where the second equality follows by definition of $(S, \phi)$ being an Eilenberg-Moore algebra, and the third equality by the fact that $\phi: ST \longrightarrow S$ is a natural transformation. Furthermore, since $\eta_D: D \longrightarrow TD$ in $\mathbf{D}$ is the identity morphism $D \longrightarrow D$ in $\mathbf{D}_{T}$,
\begin{align*}
    \overline{S}(1_D) = \phi_D S(\eta_D) = 1_{SD}
\end{align*}
again by definition of $(S, \phi)$ being an Eilenberg-Moore algebra, and so $\overline{S}$ is indeed a functor. Finally, note that since $(S, \phi)$ is Frobenius-Eilenberg-Moore algebra for $\mathsf{DagCat}(T, \mathbf{X})$, one has that for each $D$ in $\mathbf{D}$,
\begin{align}\label{FEM-KL}
    \phi_{TD} S(\mu^\dagger_D) = S(\mu_D) \phi^\dagger_{TD}
\end{align}
Therefore, for $f: C \longrightarrow TD$ in $\mathbf{D}$,
\begin{align*}
    \overline{S}(f^\dagger) &= \phi_C ST(f^\dagger) S(\mu^\dagger_D) S(\eta_D) \\
    &= S(f^\dagger) \phi_{TD} S(\mu^\dagger_D) S(\eta_D) \\
    &\stackrel{(\ref{FEM-KL})}{=} S(f^\dagger) S(\mu_D) \phi^\dagger_{TD} S(\eta_D) \\
    &= S(f^\dagger) S(\mu_D) ST(\eta_D) \phi^\dagger_D \\
    &= S(f^\dagger) \phi^\dagger_D \\
    &= \big(\phi_D S(f)\big)^\dagger = \overline{S}(f)^\dagger
\end{align*}
Sending a homomorphism of Eilenberg-Moore algebras $\psi: (P, \rho) \longrightarrow (S, \phi)$ to its underlying natural transformation $\psi: \overline{P} \longrightarrow \overline{S}$ now determines a dagger functor
\begin{align}\label{dagger-iso-right}
    \mathsf{FEM}\big(\mathsf{DagCat}(\mathbf{D}, \mathbf{X}), \mathsf{DagCat}(T, \mathbf{X})\big) \longrightarrow \mathsf{DagCat}\big(\mathbf{D}_{T}, \mathbf{X}\big)
\end{align}
These two dagger functors (\ref{dagger-iso-left}) and (\ref{dagger-iso-right}) determine an isomorphism of dagger categories
\begin{align*}
    \mathsf{DagCat}\big(\mathbf{D}_{T}, \mathbf{X}\big) \cong \mathsf{FEM}\big(\mathsf{DagCat}(\mathbf{D}, \mathbf{X}), \mathsf{DagCat}(T, \mathbf{X})\big)
\end{align*}
$2$-natural in the arguments.
\end{proof}

Our next contribution is to explicitly construct the \textit{free completion} under Frobenius-Eilenberg–Moore objects of a dagger $2$-category $\mathcal{D}$. What we mean by this free completion will be clear from Theorem \ref{freeCompletionFEM} below. Informally, however, it will manifest a `dagger-enriched' version of the closure $\overline{\mathcal{K}}$ of a $2$-category $\mathcal{K}$ in $[\mathcal{K}^{\text{op}}, \mathsf{Cat}]$ under Eilenberg-Moore objects, as detailed in \cite[Section 4]{Str76}. Rather than attempting to extend to the dagger context the sophisticated machinery of \cite{Str76}, we more directly approach the current situation, while equally following very closely the similar argument in \cite{LS02}.

Given a dagger $2$-category $\mathcal{D}$, each dagger Frobenius monad $(F, \phi)$ in $[\mathcal{D}^\text{op}, \mathsf{DagCat}]$ has a Frobenius-Kleisli object, which we denote $\mathsf{FK}(F, \phi)$. Indeed, Theorem \ref{DagCatHasFK} shows that there exists a dagger $2$-functor $\mathsf{FK}: \mathsf{DFMnd}(\mathsf{DagCat}) \longrightarrow \mathsf{DagCat}$ -- which is in fact a left $2$-adjoint of the inclusion dagger $2$-functor $I: \mathsf{DagCat} \longrightarrow \mathsf{DFMnd}(\mathsf{DagCat})$ -- and so one constructs a dagger $2$-functor $\mathsf{FK}(F, \phi): \mathcal{D}^\text{op} \longrightarrow \mathsf{DagCat}$ in the obvious fashion of specifying $0$-cells, $1$-cells and $2$-cells in $\mathsf{DFMnd}(\mathsf{DagCat})$ determined by the pointwise values of $(F, \phi)$, and then taking their images under the dagger $2$-functor $\mathsf{FK}: \mathsf{DFMnd}(\mathsf{DagCat}) \longrightarrow \mathsf{DagCat}$. Finally, since we now have, for each $D$ in $\mathcal{D}$, a $2$-natural isomorphism of dagger categories
\begin{align*}
    \mathsf{DagCat}\big(\mathsf{FK}(FD, \phi_D), SD\big) \cong \mathsf{FEM}\big(\mathsf{DagCat}(FD, SD), \mathsf{DagCat}(\phi_D, SD)\big)
\end{align*}
we surely have a $2$-natural isomorphism of dagger categories
\begin{align*}
    [\mathcal{D}^\text{op}, \mathsf{DagCat}](\mathsf{FK}(F, \phi), S) \cong \mathsf{FEM}([\mathcal{D}^\text{op}, \mathsf{DagCat}](F, S), [\mathcal{D}^\text{op}, \mathsf{DagCat}](\phi, S))
\end{align*}
for each dagger $2$-functor $S: \mathcal{D}^\text{op} \longrightarrow \mathsf{DagCat}$. 

Now, we proceed by a familiar transfinite process of, starting with the collection of all representable dagger $2$-functors $\mathcal{D}(-, D)$ in $[\mathcal{D}^\text{op}, \mathsf{DagCat}]$ and adding to this collection at each step thereafter, all Frobenius-Kleisli objects of dagger Frobenius monads involving objects of the collection at the previous step. Since the argument presented in \cite{LS02} boils down to the fact that the free functor $\mathcal{D}(X, D) \longrightarrow \mathcal{D}(X, D)_T$ to the Kleisli category for a monad $T$ on $\mathcal{D}(X, D)$ is bijective on objects, the same argument applies \textit{mutis mutandis} in our dagger case, so that this transfinite process in fact also terminates after the first step.

In conclusion, taking the replete full dagger sub-$2$-category of $[\mathcal{D}^\text{op}, \mathsf{DagCat}]$ of objects resulting from the single step of this process produces a dagger $2$-category having Frobenius-Kleisli objects. Furthermore, since each representable $\mathcal{D}(-, D)$ is itself a Frobenius-Kleisli object for a dagger Frobenius monad on a representable (for example, the identity monad on $\mathcal{D}(-, D)$), every object of this dagger $2$-category is a Frobenius-Kleisli object for a dagger Frobenius monad on a representable. We shall denote this dagger $2$-category by $\mathsf{FK}(\mathcal{D})$.

A simplification is possible which allows us to give an explicit description of $\mathsf{FK}(\mathcal{D})$.

\begin{proposition}\label{propHomCatFK}
Each $1$-cell $\mathsf{FK}\big(\mathcal{D}(-, D), \mathcal{D}(-, t)\big) \longrightarrow \mathsf{FK}\big(\mathcal{D}(-, C), \mathcal{D}(-, s)\big)$ in $\mathsf{FK}(\mathcal{D})$ is a pair $(f, \sigma)$ in which $f: D \longrightarrow C$ is a $1$-cell in $\mathcal{D}$ and $\sigma: ft \longrightarrow sf$ a $2$-cell in $\mathcal{D}$ which make the following diagrams
    \begin{equation}\label{eq:morphOfMndop}
        \bfig
        \node TAL(0,650)[sft]
        \node T2AR(800,650)[ssf]
        \node P(-450,325)[ftt]
        \node T2AL(0,0)[ft]
        \node TAR(800,0)[sf]
        \arrow/->/[TAL`T2AR;s \sigma]
        \arrow/->/[P`TAL;\sigma t]
        \arrow|b|/->/[P`T2AL;f \mu^t]
        \arrow|r|/->/[T2AR`TAR;\mu^s f]
        \arrow|r|/->/[T2AL`TAR;\sigma]
        \node TAL2(2000,650)[ssf]
        \node T2AR2(2800,650)[sf]
        \node P2(1550,325)[sft]
        \node T2AL2(2000,0)[ftt]
        \node TAR2(2800,0)[ft]
        \arrow/->/[TAL2`T2AR2;\mu^s f]
        \arrow/->/[P2`TAL2;s\sigma]
        \arrow|b|/->/[P2`T2AL2;\sigma^\dagger t]
        \arrow|r|/->/[T2AR2`TAR2;\sigma^\dagger]
        \arrow|r|/->/[T2AL2`TAR2;f\mu^t]
        \node A1IB1(675,-400)[ft]
        \node 1AI1B(1675,-400)[sf]
        \node AB(1175,-800)[f]
        \arrow/->/[A1IB1`1AI1B;\sigma]
        \arrow|l|/->/[AB`A1IB1;f\eta^t]
        \arrow|r|/->/[AB`1AI1B;\eta^sf]
        \efig
    \end{equation}
commute. Furthermore, each $2$-cell in $\mathsf{FK}(\mathcal{D})$ between such $1$-cells $(f, \sigma)$, $(g, \gamma)$ is a $2$-cell $\alpha: f \longrightarrow sg$ in $\mathcal{D}$ such that the diagram

\begin{equation}\label{op2CellMorphFEM}
        \bfig
        \node TAL2(1700,600)[sf]
        \node T2AR(2300,600)[ssg]
        \node P2(1100,600)[ft]
        \node T2AL3(1100,0)[sgt]
        \node T2AL2(1700,0)[ssg]
        \node TAR2(2300,0)[sg]
        \arrow/->/[TAL2`T2AR;s \alpha]
        \arrow/->/[P2`TAL2;\sigma]
        \arrow/->/[P2`T2AL3;\alpha t]
        \arrow|b|/->/[T2AL3`T2AL2;s\gamma]
        \arrow|r|/->/[T2AR`TAR2;\mu^sg]
        \arrow|b|/->/[T2AL2`TAR2;\mu^sg]
        \efig
\end{equation}
commutes.
\end{proposition}

\begin{proof}
We proceed by similar arguments presented in \cite{LS02}. For the objects $\mathsf{FK}\big(\mathcal{D}(-, D), \mathcal{D}(-, t)\big)$ and $\mathsf{FK}\big(\mathcal{D}(-, C), \mathcal{D}(-, s)\big)$ in $\mathsf{FK}(\mathcal{D})$, determined, respectively, by the dagger Frobenius monads $(D, t)$ and $(C, s)$ in $\mathcal{D}$, a $1$-cell
\begin{align}\label{1cells}
    \mathsf{FK}\big(\mathcal{D}(-, D), \mathcal{D}(-, t)\big) \longrightarrow \mathsf{FK}\big(\mathcal{D}(-, C), \mathcal{D}(-, s)\big)
\end{align}
is a FEM-algebra for the dagger Frobenius monad 
\begin{align}\label{proofDfmnd1}
    \mathsf{FK}\big(\mathcal{D}\big)\Big(\mathcal{D}(-, t),\mathsf{FK}\big(\mathcal{D}(-, C), \mathcal{D}(-, s)\big)\Big)
\end{align}
on the dagger category
\begin{align}\label{proofDfmnd2}
    \mathsf{FK}\big(\mathcal{D}\big)\Big(\mathcal{D}(-, D), \mathsf{FK}\big(\mathcal{D}(-, C), \mathcal{D}(-, s)\big)\Big)    
\end{align}
by the definition of Frobenius-Kleisli objects. By the Yoneda lemma for dagger $2$-categories, (\ref{proofDfmnd2}) is $2$-naturally isomorphic to $\mathsf{FK}\big(\mathcal{D}(D, C), \mathcal{D}(D, s)\big)$, while the dagger Frobenius monad corresponding to (\ref{proofDfmnd1}) is denoted $\mathsf{FK}\big(\mathcal{D}(t, C), \mathcal{D}(t, s)\big)$. By a similar argument, a $2$-cell between $1$-cells (\ref{1cells}) is simply a morphism of Eilenberg-Moore algebras between the corresponding FEM-algebras. That is, there is an isomorphism of dagger categories between the dagger category $\mathsf{FK}(\mathcal{D})\Big(\mathsf{FK}\big(\mathcal{D}(-, D), \mathcal{D}(-, t)\big), \mathsf{FK}\big(\mathcal{D}(-, C), \mathcal{D}(-, s)\big)\Big)$ and the dagger category $\mathsf{FEM}\Big(\mathsf{FK}\big(\mathcal{D}(D, C), \mathcal{D}(D, s)\big), \mathsf{FK}\big(\mathcal{D}(t, C), \mathcal{D}(t, s)\big)\Big)$ which is $2$-natural in the arguments.

Now, the dagger category $\mathsf{FK}\big(\mathcal{D}(D, C), \mathcal{D}(D, s)\big)$ has as objects $1$-cells $f: D \longrightarrow C$ in $\mathcal{D}$, and as morphisms $2$-cells $\sigma: f \longrightarrow sg$ in $\mathcal{D}$. Composition is given by the usual Kleisli composition. Turning to the dagger Frobenius monad $\mathsf{FK}\big(\mathcal{D}(t, C), \mathcal{D}(t, s)\big)$, its (dagger) endofunctor part acts on objects by $f \longmapsto ft$ and on morphisms $(\sigma: f \longrightarrow sg) \longmapsto (\sigma t : ft \longrightarrow sgt)$. The component at some $g: D \longrightarrow C$ of the multiplication part of this dagger Frobenius monad is given by $\eta^s gt \cdot g\mu^t$. Likewise, the component at $g: D \longrightarrow C$ of the unit part is given by $\eta^s gt \cdot g\eta^t$. 

Therefore, a $1$-cell (\ref{1cells}) is a pair $(f, \sigma)$ in which $f: D \longrightarrow C$ is a $1$-cell in $\mathcal{D}$, and $\sigma: ft \longrightarrow sf$ a $2$-cell in $\mathcal{D}$ satisfying the associative, unit and Frobenius laws for a FEM-algebra for the dagger Frobenius monad $\mathsf{FK}\big(\mathcal{D}(t, C), \mathcal{D}(t, s)\big)$ -- the first two laws of which give the top-left and bottom diagrams of (\ref{eq:morphOfMndop}). 

It remains only to calculate the Frobenius law diagram for $(f, \sigma)$. By \cite[Lemma 6.8]{HK16}, this is exactly the commutativity of the diagram 
\[
    \bfig
    \node A(0,600)[ft]
    \node B(0,0)[sf]
    \node C(900,600)[sftt]
    \node D(900,0)[sft]
    \arrow/->/[A`C;\tau]
    \arrow|b|/->/[B`D;\rho]
    \arrow/->/[A`B;\sigma]
    \arrow|r|/->/[C`D;\mu^s ft \cdot s \sigma t]
    \efig
\]
in which $\tau = sf{\mu^t}^\dagger \cdot s{\eta^s}^\dagger ft \cdot {\mu^s}^\dagger ft \cdot \eta^s ft$ and $\rho = \mu^s ft \cdot ss\sigma^\dagger \cdot s{\mu^s}^\dagger f \cdot s \eta^s f$. The top path is
\begin{align*}
    \mu^s ft \cdot s \sigma t \cdot \tau &= \mu^s ft\cdot s \sigma t \cdot sf{\mu^t}^\dagger \cdot s{\eta^s}^\dagger ft \cdot {\mu^s}^\dagger ft \cdot \eta^s ft \\ 
    &= \mu^s ft \cdot s \sigma t \cdot sf{\mu^t}^\dagger \cdot \eta^s ft \\
    &= \mu^s ft \cdot s (\sigma t \cdot f{\mu^t}^\dagger) \cdot \eta^s ft \\
    &= \mu^s ft \cdot \eta^s sft \cdot \sigma t \cdot f{\mu^t}^\dagger \\
    &= \sigma t \cdot f{\mu^t}^\dagger
\end{align*}
while, using the Frobenius law for the dagger Frobenius monad $(C, s)$, for the bottom path we have
\begin{align*}
    \rho \cdot \sigma &=\mu^s ft \cdot ss\sigma^\dagger \cdot s{\mu^s}^\dagger f \cdot s \eta^s f \cdot \sigma \\
    &= s\sigma^\dagger \cdot \mu^s sf \cdot s{\mu^s}^\dagger f \cdot s \eta^s f \cdot \sigma \\
    &= s\sigma^\dagger \cdot s\mu^s f \cdot {\mu^s}^\dagger sf \cdot s \eta^s f \cdot \sigma \\
    &= s\sigma^\dagger \cdot s\mu^s f \cdot ss\eta^s f \cdot {\mu^s}^\dagger f \cdot \sigma \\
    &= s\sigma^\dagger \cdot {\mu^s}^\dagger f \cdot \sigma
\end{align*}
That is, the commutativity of the Frobenius law diagram for $(f, \sigma)$ is the equality
\begin{align*}
    \sigma^\dagger \cdot {\mu^s} f \cdot s\sigma =  f{\mu^t} \cdot \sigma^\dagger t    
\end{align*}
which is exactly the top-right diagram (\ref{eq:morphOfMndop}).

Finally, as in the (possibly) non-dagger case presented in \cite{LS02}, to give a $2$-cell $\alpha: (f, \sigma) \longrightarrow (g, \gamma)$ seen as FEM-algebras for the dagger Frobenius monad $\mathsf{FK}\big(\mathcal{D}(t, C), \mathcal{D}(t, s)\big)$ is to give a $2$-cell $\alpha: f \longrightarrow sg$ in $\mathcal{D}$ satisfying (\ref{op2CellMorphFEM}). The dagger $\alpha^\dagger$, considered as a $2$-cell in $\mathsf{FK}(\mathcal{D})$, is calculated from the canonical dagger (\ref{daggerOnFK}) as the $2$-cell $s\alpha^\dagger \cdot \mu^{s \dagger} g \cdot \eta^s g : g \longrightarrow sf$ in $\mathcal{D}$.
\end{proof}

We may now take the $0$-cells of $\mathsf{FK}(\mathcal{D})$ to be dagger Frobenius monads in $\mathcal{D}$, and $1$- and $2$-cells in $\mathsf{FK}(\mathcal{D})$ to be as described in the above proposition. Furthermore, the Yoneda embedding dagger $2$-functor induces a ($2$-)fully faithful dagger $2$-functor $I: \mathcal{D} \longrightarrow \mathsf{FK}(\mathcal{D})$ whose action on $0$-cells is given by $D \longmapsto (D, 1)$, the identity dagger Frobenius monad on $D$.

Furthermore, we now define $\mathsf{FEM}(\mathcal{D}) = \mathsf{KL}(\mathcal{D}^\text{op})^\text{op}$. A $0$-cell in $\mathsf{FEM}(\mathcal{D})$ is once again a dagger Frobenius monad in $\mathcal{D}$, while $1$-cells are the same as $1$-cells in $\mathsf{DFMnd}(\mathcal{D})$. A $2$-cell $(f, \sigma) \longrightarrow (g, \gamma): (D, t) \longrightarrow (C, s)$ is a $2$-cell $\alpha: f \longrightarrow gt$ in $\mathcal{D}$ such that the diagram
\begin{equation}\label{FEMD2Cells}
        \bfig
        \node TAL2(1700,600)[ft]
        \node T2AR(2300,600)[gtt]
        \node P2(1100,600)[sf]
        \node T2AL3(1100,0)[sgt]
        \node T2AL2(1700,0)[gtt]
        \node TAR2(2300,0)[gt]
        \arrow/->/[TAL2`T2AR;\alpha t]
        \arrow/->/[P2`TAL2;\sigma]
        \arrow/->/[P2`T2AL3;s\alpha]
        \arrow|b|/->/[T2AL3`T2AL2;\gamma t]
        \arrow|r|/->/[T2AR`TAR2;g\mu^t]
        \arrow|b|/->/[T2AL2`TAR2;g\mu^t]
        \efig
\end{equation}
commutes. Again, the restricted Yoneda embedding dagger $2$-functor induces a dagger $2$-functor $I: \mathcal{D} \longrightarrow \mathsf{FEM}(\mathcal{D})$ whose action on $0$-cells is given by $D \longmapsto (D, 1)$.

\begin{example}\label{DwithFEMObjects}
Consider $\mathsf{FEM}(\mathcal{D})$ for the case that the dagger $2$-category $\mathcal{D}$ has Frobenius-Eilenberg-Moore objects. As usual, it has $0$-cells as dagger Frobenius monads in $\mathcal{D}$. Given dagger Frobenius monads $(D, t)$ and $(C, s)$ in $\mathcal{D}$, there is a bijection between the set of $1$-cells $(f, \sigma): (D, t) \longrightarrow (C, s)$ of $\mathsf{DFMnd}(\mathcal{D})$ (and hence $\mathsf{FEM}(\mathcal{D})$) and the set of pairs $(f, \overline{f})$ of $1$-cells in $\mathcal{D}$ such that the diagram
\[
        \bfig
        \node DT(0,500)[\mathsf{FEM}(D, t)]
        \node CS(900,500)[\mathsf{FEM}(C, s)]
        \node D(0,0)[D]
        \node C(900,0)[C]
        \arrow/->/[DT`CS;\overline{f}]
        \arrow/->/[DT`D;u^t]
        \arrow|r|/->/[CS`C;u^s]
        \arrow|b|/->/[D`C;f]
        \efig
\]
commutes, where $u^t$ and $u^s$ are the forgetful $1$-cells. To see this, first fix the $1$-cell $f: D \longrightarrow C$. To give a $1$-cell $\overline{f} : \mathsf{FEM}(D, t) \longrightarrow \mathsf{FEM}(C, s)$ such that the above diagram commutes is, by the definition of the FEM-object $\mathsf{FEM}(C, s)$, to give a FEM-algebra $(fu^t, \xi)$ for the dagger Frobenius monad $\mathcal{D}(\mathsf{FEM}(D, t), s)$ on the dagger category $\mathcal{D}(\mathsf{FEM}(D, t), C)$. But the adjunction $f^t \dashv u^t$ in $\mathcal{D}$ of course induces an adjunction $\mathcal{D}(u^t, C) \dashv \mathcal{D}(f^t, C)$ in $\mathsf{DagCat}$, so that there is a bijection
\begin{align*}
    \mathcal{D}(\mathsf{FEM}(D, t), C)(sfu^t, fu^t) \cong \mathcal{D}(D, C)(sf, ft)
\end{align*}
So, by \cite[Lemma 6.8]{HK16}, to give such a FEM-algebra $(fu^t, \xi)$ for the dagger Frobenius monad $\mathcal{D}(\mathsf{FEM}(D, t), S)$ is exactly to give a morphism $(f, \sigma): (D, t) \longrightarrow (C, s)$ of dagger Frobenius monads. 

In other words, $1$-cells $(D, t) \longrightarrow (C, s)$ in $\mathsf{FEM}(\mathcal{D})$ are pairs $(f, \overline{f})$ of $1$-cells in $\mathcal{D}$ satisfying $fu^t = u^s \overline{f}$.

As is true in the (possibly) non-dagger case in \cite{LS02}, a $2$-cell $(f, \overline{f}) \longrightarrow (g, \overline{g})$ in $\mathsf{FEM}(\mathcal{D})$ from $(D, t)$ to $(C, s)$ is simply a $2$-cell $\overline{f} \longrightarrow \overline{g}$ in $\mathcal{D}$. 

Next, suppose that $t\eta^t = \eta^t t$. We show that under this condition, this correspondence of $2$-cells preserves daggers. Indeed, for $1$-cells $(f, \overline{f})$, $(g, \overline{g})$ in $\mathsf{FEM}(\mathcal{D})$, to give a $2$-cell $\overline{\alpha}: \overline{f} \longrightarrow \overline{g}$ in $\mathcal{D}$ is exactly to give a $2$-cell $u^s\overline{\alpha}f^t \cdot f \eta^t = \alpha: f \longrightarrow gt$ in $\mathsf{FEM}(\mathcal{D})$. Therefore, to give the $2$-cell $\overline{\alpha}^\dagger : \overline{g} \longrightarrow \overline{f}$ in $\mathcal{D}$ is exactly to give the $2$-cell $u^s\overline{\alpha}^\dagger f^t \cdot g \eta^t: g \longrightarrow ft$. But $\alpha^\dagger$ is calculated as the $2$-cell
\begin{align*}
    (u^s\overline{\alpha}f^t \cdot f \eta^t)^\dagger t  \cdot g \mu^{t \dagger} \cdot g \eta^t = f \eta^{t \dagger} t \cdot u^s\overline{\alpha}^\dagger f^t t \cdot g \mu^{t \dagger} \cdot g \eta^t
\end{align*}
in $\mathcal{D}$. Therefore, it remains to show that $u^s\overline{\alpha}^\dagger f^t \cdot g \eta^t = f \eta^{t \dagger} t \cdot u^s\overline{\alpha}^\dagger f^t t \cdot g \mu^{t \dagger} \cdot g \eta^t$, which is the case when $t\eta^{t \dagger} = \eta^{t \dagger} t$.
\end{example}

\begin{theorem}\label{freeCompletionFEM}
Let $\mathcal{D}$ be a dagger $2$-category, and let $\mathcal{C}$ be a dagger $2$-category such that, for every dagger Frobenius monad $(C, s)$ in $\mathcal{C}$, the equality
\begin{align*}
    s\eta^s = \eta^s s
\end{align*}
holds. Then, if $\mathcal{C}$ has Frobenius-Eilenberg-Moore objects, composition with the dagger inclusion $2$-functor $I$ induces an equivalence of categories
\begin{align*}
    [\mathsf{FEM}(\mathcal{D}), \mathcal{C}]_{\mathsf{FEM}} \approx [\mathcal{D}, \mathcal{C}]
\end{align*}
between the dagger $2$-functor category $[\mathcal{D}, \mathcal{C}]$ and the full subcategory of the dagger $2$-functor category $[\mathsf{FEM}(\mathcal{D}), \mathcal{C}]_{\mathsf{FEM}}$ of dagger $2$-functors which preserve Frobenius-Eilenberg-Moore objects.
\end{theorem}

\begin{proof}
Since $\mathcal{C}$ has FEM-objects, a FEM-object-preserving dagger $2$-functor $\overline{F}: \mathsf{FEM}(\mathcal{D})$ $\longrightarrow \mathcal{C}$ extending $F: \mathcal{D} \longrightarrow \mathcal{C}$ must be defined (up to $2$-natural isomorphism) on $0$-cells by
\begin{align*}
    \overline{F}\big(D, t\big) = \mathsf{FEM}\big(FD, Ft\big)    
\end{align*}
while its action on $1$-cells and $2$-cells must be defined by the action of the composite of the dagger $2$-functor $\mathsf{FEM}(\mathcal{C}) \longrightarrow \mathcal{C}$ of Example \ref{DwithFEMObjects} with the dagger $2$-functor $\mathsf{FEM}(F) : \mathsf{FEM}(\mathcal{D}) \longrightarrow \mathsf{FEM}(\mathcal{C})$ induced by $F$.

On the other hand, these requirements can be used as a definition of such a dagger $2$-functor $\overline{F}: \mathsf{FEM}(\mathcal{D}) \longrightarrow \mathcal{C}$. Therefore, the desired extension does exist and is unique up to a $2$-natural isomorphism.
\end{proof}

\begin{proposition}
If the inclusion dagger $2$-functor functor $I :\mathcal{D} \longrightarrow \mathsf{FEM}(\mathcal{D})$ has a right $2$-adjoint, then $\mathcal{D}$ has Frobenius-Eilenberg-Moore objects.
\end{proposition}

\begin{proof}
We prove the dual result for Frobenius-Kleisli objects. Suppose $I$ has a left $2$-adjoint $L: \mathsf{FK}(\mathcal{D}) \longrightarrow \mathcal{D}$. Then, for any dagger Frobenius monad $(D, t)$ in $\mathcal{D}$,
\begin{align*}
    \mathcal{D}\big(L\big(\mathsf{FK}\big(\mathcal{D}(-, D), \mathcal{D}(-, t)\big)\big), X\big) &\cong \mathsf{FK}(\mathcal{D})\big(\mathsf{FK}\big(\mathcal{D}(-, D), \mathcal{D}(-, t)\big), I(X)\big) \\
    &= \mathsf{FK}(\mathcal{D})\big(\mathsf{FK}\big(\mathcal{D}(-, D), \mathcal{D}(-, t)\big), \mathcal{D}(-, X)\big) \\
    &\cong \mathsf{FEM}\big(\mathsf{FK}(\mathcal{D})\big(\mathcal{D}(-, D), \mathcal{D}(-, X)\big), \mathsf{FK}(\mathcal{D})\big(\mathcal{D}(-, t), \mathcal{D}(-, X)\big)\big) \\
    &\cong \mathsf{FEM}\big(\mathcal{D}\big(D, X\big), \mathcal{D}\big(t, X\big)\big)
\end{align*}
Therefore, the object $L\big(\mathsf{FK}\big(\mathcal{D}(-, D), \mathcal{D}(-, t)\big)\big)$ is a Frobenius-Kleisli object for $(D, t)$.
\end{proof}

\section{Dagger lax functors and dagger lax-limits}\label{dagLaxFunctors}

In this section, we extend the notion of a lax functor between $2$-categories to the dagger context. This will allow us to describe the universal properties of FEM-objects in Section \ref{sec-FEM-results} as dagger analogues of lax-limits of lax functors.

\begin{definition}\label{dagLaxFunct}
Given dagger $2$-categories $\mathcal{D}$, $\mathcal{C}$, a lax functor $F:  \mathcal{D} \longrightarrow \mathcal{C}$ -- having families
\begin{align*}
    \gamma_{A, B, C}: c^{\mathcal{C}} \cdot (F_{A, B} \times F_{B, C}) \longrightarrow F_{A, C} \cdot c^{\mathcal{D}}    
\end{align*}
and
\begin{align*}
    \delta_A : u^\mathcal{C} \longrightarrow F_{A, A} \cdot u^\mathcal{D}
\end{align*}
of `comparison' natural transformations -- is a dagger lax functor when, for each $A$, $B$ in $\mathcal{D}$, the functors $F_{A, B} : \mathcal{D}(A, B) \longrightarrow \mathcal{C}(FA, FB)$ are dagger functors, and the families $\gamma$ and $\delta$ additionally satisfy the \emph{Frobenius axiom}: For every triple of arrows
\[
    \bfig
        \node A(0,0)[A]
        \node B(500,0)[B]
        \node C(1000,0)[C]
        \node D(1500,0)[D]
        \arrow/->/[A`B;f]
        \arrow/->/[B`C;g]
        \arrow/->/[C`D;h]
    \efig
\]
in $\mathcal{D}$, the following diagram in $\mathcal{C}$
\begin{equation}\label{FrobAxiom}
    \bfig
        \node h1gf(0,400)[F(h) \cdot F(g \cdot f)]
        \node hgf(1200,400)[F(h) \cdot F(g) \cdot F(f)]
        \node 1hgf(0,0)[F(h \cdot g \cdot f)]
        \node hg1f(1200,0)[F(h \cdot g) \cdot F(f)]
        \arrow/->/[h1gf`hgf;1_{Fh} * \gamma^\dagger_{f, g}]
        \arrow|r|/->/[hgf`hg1f;\gamma_{g, h} * 1_{Ff}]
        \arrow/->/[h1gf`1hgf;\gamma_{g \cdot f, h}]
        \arrow|b|/->/[1hgf`hg1f;\gamma^\dagger_{f, h \cdot g}]
    \efig
\end{equation}
commutes, where $*$ indicates the horizontal composition of $2$-cells in $\mathcal{C}$ and, for simplicity, we have written $\gamma_{f, g}$ instead of $\big(\gamma_{A, B, C}\big)_{(f, g)}$.
\end{definition}

Let us clarify that the composite
\[
    \bfig
        \node D(0,0)[\mathcal{D}]
        \node C(500,0)[\mathcal{C}]
        \node B(1000,0)[\mathcal{B}]
        \arrow/->/[D`C;F]
        \arrow/->/[C`B;G]
    \efig
\]
dagger lax functor is indeed well-defined. For, given two such dagger lax functors, the composite family $\gamma^{GF}$ is determined via the pasting operation
\[
    \bfig
        \node DAxB(0,400)[\mathcal{D}(A, B) \times \mathcal{D}(B, C)]
        \node DAC(1600,400)[\mathcal{D}(A, C)]
        \node CFAxFB(0,0)[\mathcal{C}\big(FA, FB\big) \times \mathcal{C}\big(FB, FC\big)]
        \node CFAFC(1600,0)[\mathcal{C}\big(FA, FC\big)]
        \node BGFAxGFB(0,-400)[\mathcal{B}\big(GFA, GFB\big) \times \mathcal{B}\big(GFB, GFC\big)]
        \node BGFAGFC(1600,-400)[\mathcal{B}\big(GFA, GFC\big)]
        \arrow/->/[DAxB`DAC;c^\mathcal{D}]
        \arrow/->/[DAxB`CFAxFB;F_{A, B} \times F_{B, C}]
        \arrow|r|/->/[DAC`CFAFC;F_{A, C}]
        \arrow/->/[CFAxFB`CFAFC;c^\mathcal{C}]
        \arrow/->/[CFAxFB`BGFAxGFB;G_{FA, FB} \times G_{FB, FC}]
        \arrow/->/[BGFAxGFB`BGFAGFC;c^\mathcal{B}]
        \arrow|r|/->/[CFAFC`BGFAGFC;G_{FA, FC}]
        \morphism(900,150)/=>/<0,200>[`;\gamma^F]
        \morphism(900,-250)/=>/<0,200>[`;\gamma^G]
    \efig
\]
That is, for $f: A \longrightarrow B$ and $g: B \longrightarrow C$ in $\mathcal{D}$, the composite $\gamma^{GF}$ comparison family is given by
\begin{align*}
    \gamma^{GF}_{f, g} = G(\gamma^F_{f, g}) \cdot \gamma^G_{Ff, Fg}
\end{align*}
Then with $f$, $g$ as above, and $h: C \longrightarrow D$ in $\mathcal{D}$, the following diagram in $\mathcal{C}$
\[
    \bfig
        \node 1(0,400)[GF(h) \cdot GF(g \cdot f)]
        \node 2(1500,400)[GF(h) \cdot G\big(F(g) \cdot F(f)\big)]
        \node 4(0,0)[G\big(F(h) \cdot F(g \cdot f)\big)]
        \node 5(1500,0)[G\big(F(h) \cdot F(g) \cdot F(f)\big)]
        \arrow/->/[1`2;1_{GFh} * G\big(\gamma^{F\dagger}_{f, g}\big)]
        \arrow/->/[2`5;\gamma^G_{Fg \cdot Ff, Fh}]
        \arrow/->/[1`4;\gamma^G_{F(g \cdot f), Fh}]
        \arrow/->/[4`5;G\big(1_{Fh} * \gamma^{F\dagger}_{f, g}\big)]
        \node 3(3000,400)[GF(h) \cdot GF(g) \cdot GF(f)]
        \node 6(3000,0)[G\big(F(h) \cdot F(g)\big) \cdot GF(f)]
        \arrow/->/[2`3;1_{GFh} * \gamma^{G\dagger}_{Ff, Fg}]
        \arrow|r|/->/[3`6;\gamma^G_{Fg, Fh} * 1_{GFf}]
        \arrow/->/[5`6;\gamma^{G\dagger}_{Ff, Fh \cdot Fg}]
        \node 7(0,-400)[GF(h \cdot g \cdot f)]
        \node 8(1500,-400)[G\big(F(h \cdot g) \cdot F(f)\big)]
        \arrow/->/[5`8;G\big(\gamma^F_{g, h} * 1_{Ff}\big)]
        \arrow/->/[4`7;G\big(\gamma^F_{g \cdot f, h}\big)]
        \arrow|r|/->/[7`8;G\big(\gamma^{F\dagger}_{f, h \cdot g}\big)]
        \node 9(3000,-400)[GF(h \cdot g) \cdot GF(f)]
        \arrow|r|/->/[6`9;G\big(\gamma^F_{g, h}\big) * 1_{GFf}]
        \arrow|r|/->/[8`9;\gamma^{G\dagger}_{Ff, F(h \cdot g)}]
    \efig
\]
commutes. Therefore, one easily sees that $\gamma^{GF}$ does indeed satisfy the Frobenius axiom (\ref{FrobAxiom}) of Definition \ref{dagLaxFunct}.

\begin{example}
For a dagger $2$-category $\mathcal{D}$, \cite[Lemma 5.4]{HV19} shows that a dagger lax functor $\mathbf{1} \longrightarrow \mathcal{D}$ from the terminal dagger $2$-category $\mathbf{1}$ to $\mathcal{D}$ is exactly a dagger Frobenius monad $(D, t)$ in $\mathcal{D}$. Moreover, each dagger $2$-functor $F : \mathcal{D} \longrightarrow \mathcal{C}$ is of course a dagger lax functor, in which the comparison families $\gamma^F$ and $\delta^F$ are simply the identity families of natural transformations. Since dagger lax functors compose, this provides an immediate proof of the fact that dagger $2$-functors send dagger Frobenius monads to dagger Frobenius monads.
\end{example}

\begin{definition}
Consider two dagger lax functors $F, G: \mathcal{D} \longrightarrow \mathcal{C}$ between dagger $2$-categories $ \mathcal{D}$, $\mathcal{C}$. A lax-natural transformation $\alpha: F \longrightarrow G$ -- having a family
\begin{align*}
    \tau_{A, B}: \mathcal{C}(\alpha_A, 1) \cdot G_{A, B} \longrightarrow \mathcal{C}(1, \alpha_B) \cdot F_{A, B}    
\end{align*}
of natural transformations -- is a dagger lax-natural transformation when $\tau$ satisfies the following additional coherence axiom: For every pair of arrows
\[
    \bfig
        \node A(0,0)[A]
        \node B(500,0)[B]
        \node C(1000,0)[C]
        \arrow/->/[A`B;f]
        \arrow/->/[B`C;g]
    \efig
\]
in $\mathcal{D}$, the following diagram in $\mathcal{C}$
\[
    \bfig
        \node 1(0,400)[G(g) \cdot G(f) \cdot \alpha_A]
        \node 2(1200,400)[G(g) \cdot \alpha_B \cdot F(f)]
        \node 3(2400,400)[\alpha_C \cdot F(g) \cdot F(f)]
        \node 4(0,0)[G(g \cdot f) \cdot \alpha_A]
        \node 5(2400,0)[\alpha_C \cdot F(g \cdot f)]
        \arrow/->/[2`1;1_{G(g)} * \tau^\dagger_f]
        \arrow/->/[2`3;\tau_g * 1_{F(f)}]
        \arrow/->/[1`4;\gamma^G_{f, g} * 1_{\alpha_A}]
        \arrow|r|/->/[3`5;1_{\alpha_C} * \gamma^F_{f, g}]
        \arrow|b|/->/[5`4;\tau^\dagger_{g \cdot f}]
    \efig
\]
commutes, where, for simplicity, we have written $\tau_f$ instead of $\big(\tau_{A, B}\big)_{f}$. Vertical composition of dagger lax-natural transformations is defined as for usual lax-natural transformations: given a dagger lax functor $H: \mathcal{D} \longrightarrow \mathcal{C}$ and a dagger lax-natural transformation $\beta: G \longrightarrow H$, the composite dagger lax-natural transformation $\delta = \beta \cdot \alpha : F \longrightarrow H$ is defined by the family of $1$-cells
\begin{align*}
    \big(\delta_A = \beta_A \cdot \alpha_A: F(A)\longrightarrow H(A)\big)_{A \in \mathcal{D}}
\end{align*}
in $\mathcal{C}$, and the family of $2$-cells
\begin{align*}
    \Big(\tau^\delta_f = \big(1_{\beta_B} * \tau^\alpha_f\big) \cdot \big(\tau^\beta_f * 1_{\alpha_A}\big) : H(f) \cdot \delta_A \longrightarrow \delta_B \cdot F(f)\Big)_{f \in \mathcal{D}(A, B)}
\end{align*}
in $\mathcal{C}$.
\end{definition}

\begin{definition}
Consider two dagger lax functors $F, G: \mathcal{D} \longrightarrow \mathcal{C}$ between dagger $2$-categories $ \mathcal{D}$, $\mathcal{C}$, and two dagger lax-natural transformations $\alpha, \beta: F \longrightarrow G$. A modification $\Xi : \alpha \rightsquigarrow \beta$ of the underlying lax-natural transformations is a dagger modification when the following additional property is satisfied: for every parallel pair $f, g: A \longrightarrow B$ of $1$-cells in $\mathcal{D}$ and every $2$-cell $\phi: f \longrightarrow g$ in $\mathcal{D}$, the following diagram in $\mathcal{C}$
\[
    \bfig
        \node 1(0,400)[G(f) \cdot \alpha_A]
        \node 2(900,400)[G(g) \cdot \beta_A]
        \node 3(0,0)[\alpha_B \cdot F(f)]
        \node 4(900,0)[\beta_B \cdot F(g)]
        \arrow/->/[2`1;G(\phi)^\dagger * \Xi^\dagger_A]
        \arrow|r|/->/[2`4;\tau^\beta_g]
        \arrow/->/[1`3;\tau^\alpha_f]
        \arrow|b|/->/[4`3;\Xi^\dagger_B * F(\phi)^\dagger]
    \efig
\]
commutes. The vertical and horizontal composition of modifications is defined as for usual modifications. Furthermore, the dagger on $2$-cells in $\mathcal{C}$ induces a dagger on dagger modifications.
\end{definition}

\begin{example}
We have already seen that a dagger lax functor $T : \mathbf{1} \longrightarrow \mathcal{D}$ is a dagger Frobenius monad $(D, t)$ in $\mathcal{D}$. Given another dagger lax functor $S : \mathbf{1} \longrightarrow \mathcal{D}$, a dagger lax-natural transformation $F: T \longrightarrow S$ is exactly a morphism of dagger Frobenius monads $(D, t) \longrightarrow (C, s)$ in $\mathcal{D}$. Given another such dagger lax-natural transformation $G : T \longrightarrow S$, a dagger modification $F \rightsquigarrow G$ is exactly a morphism in $\mathsf{DFMnd}(\mathcal{D})\big((D, t), (C, s)\big)$ from the morphism of dagger Frobenius monads corresponding to $F$, to the morphism of dagger Frobenius monads corresponding to $G$.
\end{example}

\begin{definition}
For dagger $2$-categories $\mathcal{D}$, $\mathcal{C}$, let $\mathsf{DagLax}_{\mathcal{D}, \mathcal{C}}$ denote the dagger $2$-category of dagger lax functors $\mathcal{D} \longrightarrow \mathcal{C}$, dagger lax-natural transformations between them, and dagger modifications between dagger lax-natural transformations. Let $\Delta_C: \mathcal{D} \longrightarrow \mathcal{C}$ denote the constant dagger $2$-functor on an object $C$ in $\mathcal{C}$. The dagger lax-limit of a dagger lax functor $F: \mathcal{D} \longrightarrow \mathcal{C}$, if it exists, is a pair $(L, \pi)$ where $L$ is an object of $\mathcal{C}$ and $\pi: \Delta_L \longrightarrow F$ is a dagger lax-natural transformation such that, for each object $C$ in $\mathcal{C}$, the dagger functor
\begin{align*}
    \mathcal{C}(C, L) \longrightarrow \mathsf{DagLax}_{\mathcal{D}, \mathcal{C}}[\Delta_C, F]
\end{align*}
of composition with $\pi$ is an isomorphism of dagger categories, $2$-natural in $C$.
\end{definition}

\begin{example}
Suppose a dagger Frobenius monad $(D, t)$ in a dagger $2$-category $\mathcal{D}$ has a Frobenius-Eilenberg-Moore object. The dagger lax-limit of $(D, t)$, considered as a dagger lax functor $\mathbf{1} \longrightarrow \mathcal{D}$, is the pair $\big(\mathsf{FEM}(D, t), \pi\big)$, where $\pi = (u^t, \xi)$ is the pair as defined below Theorem \ref{hk16}. For, to say that $(L, \pi)$ is a dagger lax-limit of $(D, t)$ -- when it exists -- is to give a pair $\pi = (h, \sigma)$ with $h: L \longrightarrow D$ a $1$-cell and $\sigma : th \longrightarrow h$ a $2$-cell in $\mathcal{D}$ such that $(h, \sigma): (L, 1) \longrightarrow (D, t)$ is a morphism of dagger Frobenius monads, and such that the following universal property is satisfied: for any $C$ in $\mathcal{D}$, $1$-cell $g: C \longrightarrow D$ and $2$-cell $\gamma : tg\longrightarrow g$ in $\mathcal{D}$ such that $(g, \gamma): (C, 1) \longrightarrow (D, t)$ is a morphism of dagger Frobenius monads, there exists a unique $1$-cell $n: C \longrightarrow L$ such that $hn = g$ and $\sigma n = \gamma$. But, by Proposition \ref{prop2} this is exactly to say that $L$ is a Frobenius-Eilenberg-Moore object for $(D, t)$.
\end{example}
\refs

\bibitem [HK, 2016]{HK16} C. Heunen and M. Karvonen.  Monads on dagger categories. \emph{Theory and Applications of Categories}, 31(35):1016–1043, 2016

\bibitem [HK, 2015]{HK15} C. Heunen and M. Karvonen.  Reversible Monadic Computing. \emph{Electronic Notes in Theoretical Computer Science}, 319:217-237, 2015

\bibitem[HV, 2019]{HV19} C. Heunen and J. Vicary, Categories for Quantum Theory: An Introduction. \emph{Oxford University Press}, Oxford Graduate Texts in Mathematics, 2019.

\bibitem[Karvonen, 2019]{Kar19} M. Karvonen, PhD Thesis: The Way of the Dagger. \emph{University of Edinburgh}, 2019.

\bibitem[Kelly, 2005]{Kel} G. M. Kelly. Basic Concepts of Enriched Category Theory. \emph{Reprints in Theory and Applications of Categories}, 1(10):1-136, 2005

\bibitem[Lauda, 2006]{Lau} A. Lauda. Frobenius algebras and ambidextrous adjunctions. \emph{Theory and Applications of Categories}, 16(4):84–122, 2006.

\bibitem[Street, 2004]{Str04} R. Street, Frobenius monads and pseudomonoids. \emph{Journal of Mathematical Physics}, 45(10):3930–3948, 2004

\bibitem[Street, 1976]{Str76} R. Street, Limits indexed by category-valued 2-functors. \emph{Journal of Pure and Applied Algebra}, 8(2):149-181, 1976.

\bibitem[Street, 1972]{Str72} R. Street, The formal theory of monads. \emph{Journal of Pure and Applied Algebra}, 2(2):149-168, 1972.

\bibitem[LS, 2002]{LS02} R. Street and S. Lack, The formal theory of monads II. \emph{Journal of Pure and Applied Algebra}, 175(1):243-265, 2002.

\endrefs
\end{document}